\documentclass[a4paper,11pt]{article}

\usepackage{amsmath,amsthm,amssymb,mathrsfs,latexsym,amsfonts}
\usepackage{graphicx,psfrag,epsfig}
\usepackage[english]{babel}
\usepackage[latin1]{inputenc}
\usepackage{pstricks}

\newtheorem{theorem}{Theorem}[section]
\newtheorem{lemma}[theorem]{Lemma}
\newtheorem{proposition}[theorem]{Proposition}

\newtheorem{question}{Question}
\newtheorem{corollary}[theorem]{Corollary}

\newcounter{paraga}[section]
\renewcommand{\theparaga}{{\bf\arabic{paraga}.}}
\newcommand{\paraga}{\medskip \addtocounter{paraga}{1} 
\noindent{\theparaga\ } }

\begin{document}

\def\N{\mathbb N}
\def\C{\mathbb C}
\def\Q{\mathbb Q}
\def\R{\mathbb R}
\def\T{\mathbb T}
\def\A{\mathbb A}
\def\Z{\mathbb Z}
\def\demi{\frac{1}{2}}

\begin{titlepage}
\author{Abed Bounemoura~\footnote{A.Bounemoura@warwick.ac.uk, Mathematics Institute, University of Warwick}}
\title{\LARGE{\textbf{An example of instability in high-dimensional Hamiltonian systems}}}
\end{titlepage}

\maketitle

\begin{abstract}
In this article, we use a mechanism first introduced by Herman, Marco, and Sauzin to show that if a Gevrey or analytic perturbation of a quasi-convex integrable Hamiltonian system is not too small with respect to the number of degrees of freedom, then the classical exponential stability estimates do not hold. Indeed, we construct an unstable solution whose drifting time is polynomial with respect to the inverse of the size of the perturbation. A different example was already given by Bourgain and Kaloshin, with a linear time of drift but with a perturbation which is larger than ours. As a consequence, we obtain a better upper bound on the threshold of validity of exponential stability estimates.
\end{abstract}
 
\section{Introduction}

\paraga Consider a near-integrable Hamiltonian system of the form
\begin{equation*}
\begin{cases} 
H(\theta,I)=h(I)+f(\theta,I) \\
|f| < \varepsilon 
\end{cases}
\end{equation*}
with angle-action coordinates $(\theta,I) \in \T^n \times \R^n$, and where $f$ is a small perturbation, of size $\varepsilon$, in some suitable topology defined by a norm $|\,.\,|$. 

If the system is analytic and $h$ satisfies a generic condition, it is a remarkable result due to Nekhoroshev (\cite{Nek77}, \cite{Nek79}) that the action variables are stable for an exponentially long interval of time with respect to the inverse of the size of the perturbation: one has
\[ |I(t)-I_0| \leq c_1\varepsilon^b, \quad |t|\leq c_2\exp(c_3\varepsilon^{-a}), \] 
for some positive constants $c_1,c_2,c_3,a,b$ and provided that the size of the perturbation $\varepsilon$ is smaller than a threshold $\varepsilon_0$. Of course, all these constants strongly depend on the number of degrees of freedom $n$, and when the latter goes to infinity, the threshold $\varepsilon_0$ and the exponent of stability $a$ go to zero.

More precisely, in the case where $h$ is quasi-convex and the system is analytic or even Gevrey, then we know that the exponent $a$ is of order $n^{-1}$ and this is essentially optimal (see \cite{LN92}, \cite{Pos93}, \cite{MS02}, \cite{LM05} \cite{KZ09} and \cite{BM10} for more information on the optimality of the stability exponent). 

\paraga This fact was used by Bourgain and Kaloshin in \cite{BK05} to show that if the size of the perturbation is
\[ \varepsilon_n \sim e^{-n}, \]
then there is no exponential stability: they constructed unstable solutions for which the time of drift is linear with respect to the inverse of the size of the perturbation, that is
\[ |I(\tau_n)-I_0|\sim 1, \quad \tau_n \sim \varepsilon_n^{-1}.  \] 
More precisely, in the first part of \cite{BK05}, Bourgain proved this result for a specific example of Gevrey non-analytic perturbation of a quasi-convex system, then for an analytic perturbation he obtained a time $\tau_n \sim \varepsilon_n^{-1-c}$, for any $c>0$. In the second part of \cite{BK05}, using much more elaborated techniques (especially Mather theory), Kaloshin proved the above result for both Gevrey and analytic perturbation and for a wider class of integrable Hamiltonians, including convex and quasi-convex systems. 

Their motivation was the implementation of stability estimates in the context of Hamiltonian partial differential equations, which requires to understand the relative dependence between the size of the perturbation and the number of degrees of freedom. Their result indicates that for infinite dimensional Hamiltonian systems, Nekhoroshev's mechanism does not survive and that ``fast diffusion" should prevail. Of course, in their example, one cannot simply take $n=\infty$ as the size of the perturbation $\varepsilon_n \sim e^{-n}$ goes to zero and the time of instability $\tau_n \sim e^n$ goes to infinity exponentially fast with respect to $n$. A more precise interpretation concerns the threshold of validity $\varepsilon_0$ in Nekhoroshev's theorem: it has to satisfy
\[ \varepsilon_0 <\!\!< e^{-n}, \]
and so it deteriorates faster than exponentially with respect to $n$.

\paraga As was noticed by the authors in \cite{BK05}, their examples share some similarities with the mechanism introduced by Herman, Marco and Sauzin in \cite{MS02} (see also \cite{LM05} for the analytic case). In this present article, we use the approach of \cite{MS02} and \cite{LM05} to show, using simpler arguments than those contained in \cite{BK05}, that if the size of the perturbation is
\[ \varepsilon_n \sim e^{-n\ln (n\ln n)}, \]
then it is still too large to have exponential stability: we will show that one can find an unstable solution where the time of drift is polynomial, more precisely     
\[ |I(\tau_n)-I_0|\sim 1, \quad \tau_n \sim \varepsilon_n^{-n}.  \] 
As in the first part of \cite{BK05}, we will construct specific examples of Gevrey and analytic perturbations of a quasi-convex system. We refer to Theorem~\ref{thmnonpert} and Theorem~\ref{thmnonpertana} below for precise statements. Hence one can infer that the threshold $\varepsilon_0$ in Nekhoroshev's theorem further satisfies
\[ \varepsilon_0 <\!\!< e^{-n\ln (n\ln n)} <\!\!< e^{-n}, \]
and this gives another evidence that the finite dimensional mechanism of stability cannot extend so easily to infinite dimensional systems. Let us point out that our time of drift is worst than the one obtained in \cite{BK05}, but this stems from the fact that the size of our perturbation is smaller than theirs and so it is natural for the time of instability to be larger. Moreover, our exponent $n$ in the time of drift can be a bit misleading since in any cases, that is even for a linear time of drift, $\tau_n$ goes to infinity exponentially fast with $n$, so such results do not apply at all to infinite dimensional Hamiltonian systems. A natural question, which was raised by Marco, is the following.

\begin{question}
Given $\varepsilon>0$ arbitrarily small, construct an $\varepsilon_n$-perturbation of an integrable system having an unstable orbit with a time of instability $\tau_n$ such that
\[ \lim_{n\rightarrow +\infty}\varepsilon_n=\varepsilon, \quad \lim_{n\rightarrow +\infty}\tau_n<+\infty. \]
\end{question}

We believe that one can give a positive answer to this question, by using a more clever construction. However, even if one can formally let $n$ goes to infinity, by no means this implies the existence of an unstable solution for an infinite-dimensional Hamiltonian systems, which is a very difficult problem (see \cite{CKSTT} and \cite{GG10} for related results in some examples of Hamiltonian partial differential equations).  

\section{Main results}

\subsection{The Gevrey case}

\paraga Let us first state our result in the Gevrey case. Let $n\geq 3$ be the number of degrees of freedom, and given $R>0$, let $B=B_R$ be the open ball of $\R^n$ around the origin, of radius $R>0$ with respect to the supremum norm $|\,.\,|$, and $\overline{B}$ its closure. 

The phase space is $\T^n \times B$, and we consider a Hamiltonian system of the form
\[ H(\theta,I)=h(I)+f(\theta,I), \quad (\theta,I)\in \T^n \times B.  \]
Our quasi-convex integrable Hamiltonian $h$ is the simplest one, namely
\[ h(I)=\demi(I_1^2+\cdots+I_{n-1}^2)+I_n, \quad I=(I_1,\dots,I_n)\in B. \]  
Let us recall that, given $\alpha \geq 1$ and $L>0$, a function $f\in C^{\infty}(\T^n \times \overline{B})$ is $(\alpha,L)$-Gevrey if, using the standard multi-index notation, 
\[ |f|_{\alpha,L}=\sum_{l\in \N^{2n}}L^{|l|\alpha}(l!)^{-\alpha}|\partial^l f|_{C^0(\T^n \times \overline{B})} < \infty. \]
The space of such functions, with the above norm, is a Banach space that we denote by $G^{\alpha,L}(\T^n \times \overline{B})$. One can see that analytic functions correspond exactly to $\alpha=1$.

\paraga Now we can state our theorem.

\begin{theorem}\label{thmnonpert}
Let $n\geq 3$, $R>1$, $\alpha>1$ and $L>0$. Then there exist positive constants $c,\gamma,C$ and $n_0\in\N^*$ depending only on $R,\alpha$ and $L$ such that for any $n\geq n_0$, the following holds: there exists a function $f_n \in G^{\alpha,L}(\T^n \times \overline{B})$ with $\varepsilon_n=|f_n|_{\alpha,L}$ satisfying
\[e^{-2(n-2)\ln (4n\ln 2n)}\leq \varepsilon_n \leq c\, e^{-2(n-2)\ln (n\ln 2n)},\]
such that the Hamiltonian system $H_n=h+f_n$ has an orbit $(\theta(t),I(t))$ for which the estimates
\[ |I(\tau_n)-I_0|\geq 1, \quad \tau_n\leq C\left(\frac{c}{\varepsilon_n}\right)^{n\gamma}, \]
hold true.
\end{theorem}

As we have already explained, this statement gives an upper bound on the threshold of applicability of Nekhoroshev's estimates, which is an important issue when trying to use abstract stability results for ``realistic" problems, for instance for the so-called planetary problem (see \cite{Nie96}). 

So let us consider the set of Gevrey quasi-convex integrable Hamiltonians $\mathcal{H}=\mathcal{H}(n,R,\alpha,L,M,m)$ defined as follows: $h\in\mathcal{H}$ if $h\in G^{\alpha,L}(\overline{B})$ and satisfies both
\[ \forall I\in B, \quad |\partial ^k h(I)|\leq M, \quad 1\leq |k_1|+\cdots+|k_n|\leq 3, \]
and
\[ \forall I\in B, \forall v\in\R^n, \quad \nabla h(I).v=0 \Longrightarrow \nabla^2 h(I)v.v \geq m|v|^2.\]
From Nekhoroshev's theorem (see \cite{MS02} for a statement in Gevrey classes), we know that there exists a positive constant $\varepsilon_0(\mathcal{H})=\varepsilon_0(n,R,\alpha,L,M,m)$ such that the following holds: for any $h\in\mathcal{H}$, there exist positive constants $c_1,c_2,c_3,a$ and $b$ such that if 
\[ f\in G^{\alpha,L}(\T^n \times \overline{B}), \quad |f|_{\alpha,L}<\varepsilon_0(\mathcal{H}),\] 
then any solution $(\theta(t),I(t))$ of the system $H=h+f$, with $I(0)\in B_{R/2}$, satisfies
\[ |I(t)-I_0| \leq c_1\varepsilon^b, \quad |t|\leq c_2\exp(c_3\varepsilon^{-a}). \]
Then we can state the following corollary of our Theorem~\ref{thmnonpert}.

\begin{corollary}
With the previous notations, one has the upper bound
\[ \varepsilon_0(\mathcal{H})<e^{-2(n-2)\ln (4n\ln 2n)}. \]
\end{corollary}

This improves the upper bound $\varepsilon_0(\mathcal{H})<e^{-n}$ obtained in \cite{BK05} for Gevrey functions. 

\subsection{The analytic case}

\paraga Let us now state our result in the analytic case. Here $B=B_R$ is still the open ball of $\R^n$ around the origin, of radius $R>0$ with respect to the supremum norm, and we will also consider a Hamiltonian system of the form
\[ H(\theta,I)=h(I)+f(\theta,I), \quad (\theta,I)\in \T^n \times B,  \]
where
\[ h(I)=\demi(I_1^2+\cdots+I_{n-1}^2)+I_n, \quad I=(I_1,\dots,I_n)\in B. \]
Given $\rho>0$, let us introduce the space $\mathcal{A}_\rho(\T^n \times B)$ of bounded real-analytic functions on $\T^n \times B$ admitting a bounded holomorphic extension to the complex neighbourhood
\[ V_\rho=V_\rho(\T^n \times B)=\{(\theta,I)\in(\C^n/\Z^n)\times \C^{n} \; | \; |\mathcal{I}(\theta)|<\rho,\;d(I,B)< \rho\}, \]
where $\mathcal{I}(\theta)$ is the imaginary part of $\theta$ and the distance $d$ is associated to the supremum norm on $\C^n$. Such a space $\mathcal{A}_\rho(\T^n \times B)$ is obviously a Banach space with the norm
\[ |f|_\rho=|f|_{C^0(V_\rho)}=\sup_{z\in V_\rho}|f(z)|, \quad f\in\mathcal{A}_\rho(\T^n \times B). \]
Furthermore, for bounded real-analytic vector-valued functions defined on $\T^n \times B$ admitting a bounded holomorphic extension to $V_\rho$, we shall extend this norm componentwise (in particular, this applies to Hamiltonian vector fields and their time-one maps). 

\paraga Now we can state our theorem.  

\begin{theorem}\label{thmnonpertana}
Let $n\geq 4$, $R>1$, and $\sigma>0$. Then there exist positive constants $\rho,\gamma,C$ and $n_0\in\N^*$ depending only on $R$ and $\sigma$, and a constant $c_n$ that may also depends on $n$, such that for any $n\geq n_0$, the following holds: there exists a function $f_n \in \mathcal{A}_\rho(\T^n \times B)$ with $\varepsilon_n=|f_n|_{\rho}$ satisfying
\[e^{-2(n-3)\ln (4n\ln 2n)}\leq \varepsilon_n \leq c_n\, e^{-2(n-3)\ln (n\ln 2n)},\]
such that the Hamiltonian system $H_n=h+f_n$ has an orbit $(\theta(t),I(t))$ for which the estimates
\[ |I(\tau_n)-I_0|\geq 1, \quad \tau_n\leq C\left(\frac{c_n}{\varepsilon_n}\right)^{n\gamma}, \]
hold true.
\end{theorem}

In the above statement, the constant $\sigma$ has to be chosen sufficiently small but independently of the choice of $n$ and $R$ (see Proposition~\ref{LocMar}). Moreover, the theorem is slightly different than the one in the Gevrey case, since there is a constant $c_n$ depending also on $n$: this comes from the use of suspension arguments due to Kuksin and Pöschel (\cite{Kuk93}, \cite{KP94}) in the analytic case, which are more difficult than in the Gevrey case.

Here we can also define a threshold of validity in the Nekhoroshev theorem $\varepsilon_0(\mathcal{H})=\varepsilon_0(n,R,\rho,M,m)$ and state the following corollary of our Theorem~\ref{thmnonpertana}.

\begin{corollary}
With the previous notations, one has the upper bound
\[ \varepsilon_0(\mathcal{H})<e^{-2(n-3)\ln (4n\ln 2n)}. \]
\end{corollary}

This improves the upper bound $\varepsilon_0(\mathcal{H})<e^{-n}$ obtained in \cite{BK05} for analytic functions. For concrete Hamiltonians like in the planetary problem, the actual distance to the integrable system is essentially of order $10^{-3}$, hence the above corollary yields the impossibility to apply Nekhoroshev's estimates for $n>3$. 

\subsection{Some remarks}

\paraga Theorem~\ref{thmnonpert} and Theorem~\ref{thmnonpertana} are obtained from the constructions in \cite{MS02} and \cite{LM05}, but one has to choose properly the dependence with respect to $n$ of the various parameters involved. 

As the reader will see, we will use only rough estimates leading to the factor $n$ in the time of instability: this can be easily improved but we do not know if it is possible in our case (that is with a perturbation of size $e^{-n\ln (n\ln n)}$) to obtain a linear time of drift.

Let us note also we have restricted the perturbation to a compact subset of $\A^n=\T^n \times \R^n$ just in order to evaluate Gevrey or analytic norms. In fact, in both theorems the Hamiltonian vector field generated by $H_n=h+f_n$ is complete and the unstable solution $(\theta(t),I(t))$ satisfies 
\[ \lim_{t\rightarrow \pm \infty}|I(t)-I_0|=+\infty, \]
which means that it is bi-asymptotic to infinity.

\paraga Ii is important to note that our approach leads, as in the first part of \cite{BK05}, to results for an autonomous perturbation of a quasi-convex integrable system (or, equivalently, for a time-dependent time-periodic perturbation of a convex integrable system). In the second part of \cite{BK05}, for a class of convex integrable systems, Kaloshin was able to reduce the case of an autonomous perturbation to the case of a time-dependent time periodic perturbation, partly because of his more general (but more involved) approach. We could have tried to apply his general arguments to our case, but for simplicity we decided not to pursue this further.

\paraga In this text, we will have to deal with time-one maps associated to Hamiltonian flows. So given a function $H$, we will denote by $\Phi_t^H$ the time-$t$ map of its Hamiltonian flow and by $\Phi^H=\Phi^H_1$ the time-one map. We shall use the same notation for time-dependent functions $H$, that is $\Phi^H$ will be the time-one map of the Hamiltonian isotopy (the flow between $t=0$ and $t=1$) generated by $H$.  

\section{Proof of Theorem~\ref{thmnonpert}}

The proof of Theorem~\ref{thmnonpert} is contained in section~\ref{sectnonpert}, but first in section~\ref{mechanism} we recall the mechanism of instability presented in the paper \cite{MS02} (see also \cite{MS04}). 

This mechanism has two main features. The first one is that it deals with perturbation of integrable maps rather than perturbation of integrable flows, and then the latter is recovered by a suspension process. This point of view, which is only of technical matter, was already used for example in \cite{Dou88} and offers more flexibility in the construction. The second feature, which is the most important one, is that instead of trying to detect instability in a map close to integrable by means of the usual splitting estimates, we will start with a map having already ``unstable" orbits and try to embed it in a near-integrable map. This will be realized through a ``coupling lemma", which is really the heart of the mechanism. 

As we will see, the construction offers an easy and very efficient way of computing the drifting time of unstable solutions, therefore avoiding all the technicalities that are usually required for such a task. 

\subsection{The mechanism} \label{mechanism}

\paraga Given a potential function $U : \T \rightarrow \R$, we consider the following family of maps $\psi_q : \A \rightarrow \A$ defined by
\begin{equation}\label{maps1}
\psi_q(\theta,I)=\left(\theta +qI, I-q^{-1}U'(\theta + qI)\right), \quad (\theta,I)\in\A, 
\end{equation}
for $q \in \N^*$. If we require $U'(0)=-1$, for example if we choose 
\[ U(\theta)=-(2\pi)^{-1}\sin (2\pi\theta),\quad U'(\theta)=-\cos(2\pi\theta),\] 
then it is easy to see that $\psi_q(0,0)=(0,q^{-1})$ and by induction 
\begin{equation}\label{driftstand}
\psi_q^k(0,0)=(0,kq^{-1}) 
\end{equation}
for any $k \in \Z$ (see figure~\ref{drift}). After $q$ iterations, the point $(0,0)$ drifts from the circle $I=0$ to the circle $I=1$ and it is bi-asymptotic to infinity, in the sense that the sequence $\left(\psi_q^k(0,0)\right)_{k \in \Z}$ is not contained in any semi-infinite annulus of $\A$. 
\begin{figure}[t] 
\centering
\def\JPicScale{0.7}
\input{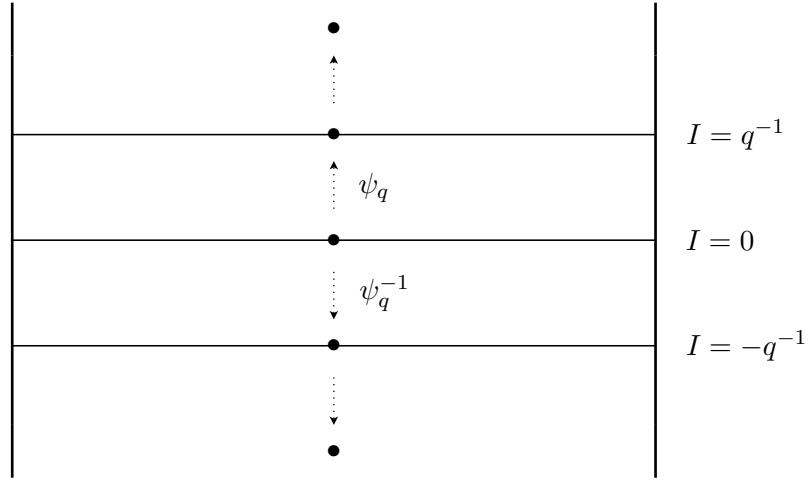}
\caption{Drifting point for the map $\psi_q$} \label{drift}
\end{figure}
Clearly these maps are exact-symplectic, but obviously they have no invariant circles and so they cannot be ``close to integrable". However, we will use the fact that they can be written as a composition of time-one maps,
\begin{equation}\label{driftstand2}
\psi_q=\Phi^{q^{-1}U} \circ \left(\Phi^{\frac{1}{2}I^2} \circ \cdots \circ \Phi^{\frac{1}{2}I^2}\right) =\Phi^{q^{-1}U} \circ \left(\Phi^{\frac{1}{2}I^2}\right)^q, 
\end{equation}
to embed $\psi_q$ in the $q^{th}$-iterate of a near-integrable map of $\A^n$, for $n \geq 2$. To do so, we will use the following ``coupling lemma", which is easy but very clever. 

\begin{lemma}[Herman-Marco-Sauzin] \label{coupling}
Let $m,m' \geq 1$, $F : \A^m \rightarrow \A^m$ and $G : \A^{m'} \rightarrow \A^{m'}$ two maps, and $f : \A^m \rightarrow \R$ and $g : \A^{m'} \rightarrow \R$ two Hamiltonian functions generating complete vector fields. Suppose there is a point $a \in \A^{m'}$ which is $q$-periodic for $G$ and such that the following ``synchronisation" conditions hold:
\begin{equation}\label{sync}
g(a)=1, \quad dg(a)=0, \quad g(G^k(a))=0, \quad dg(G^k(a))=0, \tag{S}
\end{equation} 
for $1 \leq k \leq q-1$. Then the mapping
\[ \Psi=\Phi^{f \otimes g} \circ (F \times G) : \A^{m+m'} \longrightarrow \A^{m+m'} \]
is well-defined and for all $x \in \A^m$, 
\[ \Psi^q(x,a)=(\Phi^f \circ F^q(x),a). \]
\end{lemma}  

The product of functions acting on separate variables was denoted by $\otimes$, \textit{i.e.}
\[ f \otimes g(x,x')=f(x)g(x') \quad  x \in \A^m , \; x' \in \A^{m'}. \]
Let us give an elementary proof of this lemma since it is a crucial ingredient.

\begin{proof}
First note that since the Hamiltonian vector fields $X_f$ and $X_g$ are complete, an easy calculation shows that for all $x\in\A^m$, $x'\in\A^{m'}$ and $t\in \R$, one has 
\begin{equation}\label{coupl}
\Phi_{t}^{f \otimes g}(x,x')=\left(\Phi_{t}^{g(x')f}(x),\Phi_{t}^{f(x)g}(x')\right) 
\end{equation}
and therefore $X_{f \otimes g}$ is also complete. Using the above formula and condition~(\ref{sync}), the points $(F^{k}(x),G^{k}(a))$, for $1\leq k \leq q-1$, are fixed by $\Phi^{f \otimes g}$ and hence
\[ \Psi^{q-1}(x,a)=(F^{q-1}(x),G^{q-1}(a)). \]
Since $a$ is $q$-periodic for $G$ this gives
\[ \Psi^{q}(x,a)=\Phi^{f \otimes g}(F^{q}(x),a), \]
and we end up with
\[ \Psi^{q}(x,a)=(\Phi^{f}(F^{q}(x)),a) \]
using once again~(\ref{sync}) and~(\ref{coupl}).
\end{proof}

Therefore, if we set $m=1$, $F=\Phi^{\frac{1}{2}I_1^2}$ and $f=q^{-1}U$ in the coupling lemma, the $q^{th}$-iterate $\Psi^q$ will leave the submanifold $\A \times \{a\}$ invariant, and its restriction to this annulus will coincide with our ``unstable map" $\psi_q$. Hence, after $q^2$ iterations of $\Psi$, the $I_1$-component of the point $((0,0),a) \in \A^2$ will move from $0$ to $1$. 

\paraga The difficult part is then to find what kind of dynamics we can put on the second factor to apply this coupling lemma. In order to have a continuous system with $n$ degrees of freedom at the end, we may already choose $m'=n-2$ so the coupling lemma will give us a discrete system with $n-1$ degrees of freedom.

First, a natural attempt would be to try 
\[ G=G_n=\Phi^{\frac{1}{2}I_2^2+\cdots+\frac{1}{2}I_{n-1}^2}.\]
Indeed, in this case
\[ F \times G_n=\Phi^{\frac{1}{2}I_2^2+\cdots+\frac{1}{2}I_{n-1}^2}=\Phi^{\tilde{h}} \] 
where 
\[ \tilde{h}(I_1,\dots,I_{n-1})=\frac{1}{2}(I_1^2+\cdots+I_{n-1}^2) \] 
and the unstable map $\Psi$ given by the coupling lemma appears as a perturbation of the form $\Psi=\Phi^u \circ \Phi^{\tilde{h}}$, with $u=f\otimes g$. However, this cannot work. Indeed, for $j\in\{2,\dots,n-1\}$, one can choose a $p_j$-periodic point $a^{(j)}\in\A$ for the map $\Phi^{\frac{1}{2}I_j^2}$, and then setting 
\[ a_n=(a^{(2)},\dots,a^{(n-1)})\in\A^{n-2}, \quad q_n=p_2 \cdots p_{n-1},\] 
the point $a_n$ is $q_n$-periodic for $G_n$ provided that the numbers $p_j$ are mutually prime. One can easily see that the latter condition will force the product $q_n$ to converge to infinity when $n$ goes to infinity. So necessarily the point $a_n$ gets arbitrarily close to its first iterate $G_n(a_n)$ when $n$ (and therefore $q_n$) is large: this is because $q_n$-periodic points for $G_n$ are equi-distributed on $q_n$-periodic tori. As a consequence, a function $g_n$ with the property
\[ g_n(a_n)=1, \quad g_n(G_n(a_n))=0,  \]
will necessarily have very large derivatives at $a_n$ if $q_n$ is large. Then as the size of the perturbation is essentially given by 
\[ |f\otimes g_n|=|q_{n}^{-1}U\otimes g_n|=|q_{n}^{-1}||g_n|,\] 
one can check that it is impossible to make this quantity converge to zero when the number of degrees of freedom $n$ goes to infinity.

\paraga As in \cite{MS02}, the idea to overcome this problem is the following one. We introduce a new sequence of ``large" parameters $N_n \in \N^*$ and in the second factor we consider a family of suitably rescaled penduli on $\A$ given by
\[ P_n(\theta_2,I_2)=\demi I_2^2 +N_n^{-2}V(\theta_2), \]
where $V(\theta)=-\cos 2\pi\theta$. The other factors remain unchanged, so 
\[ G_n=\Phi^{\frac{1}{2}(I_2^2+I_3^2+\cdots+I_{n-1}^2)+N_{n}^{-2}V(\theta_2)}. \]
In this case, the map $\Psi$ given by the coupling lemma is also a perturbation of $\Phi^{\tilde{h}}$ but of the form $\Psi=\Phi^u \circ \Psi^{\tilde{h}+v}$, with $v=N_n^{-2}V$. But now for the map $G_n$, due to the presence of the pendulum factor, it is now possible to find a periodic orbit with an irregular distribution: more precisely, a $q_n$-periodic point $a_n$ such that its distance to the rest of its orbit is of order $N_{n}^{-1}$, no matter how large $q_n$ is. 

\paraga Let us denote by $(p_j)_{j\geq 0}$ the ordered sequence of prime numbers and let us choose $N_n$ as the product of the $n-2$ prime numbers $\{p_{n+3},\dots,p_{2n}\}$, that is
\[ N_n=p_{n+3}p_{n+4}\cdots p_{2n}\in\N^*.\]
Our goal is to prove the following proposition.

\begin{proposition}\label{pertu} 
Let $n\geq 3$, $\alpha>1$ and $L_1>0$. Then there exist a function $g_n\in G^{\alpha,L_1}(\A^{n-2})$, a point $a_n\in\A^{n-2}$ and positive constants $c_1$ and $c_2$ depending only on $\alpha$ and $L_1$ such that if
\[ M_n=2\left[c_1N_ne ^{c_2(n-2)p_{2n}^{\frac{1}{\alpha}}}\right], \quad q_n=N_nM_n, \]
then $a_n$ is $q_n$-periodic for $G_n$ and $(g_n, G_n, a_n, q_n)$ satisfy the synchronization conditions (\ref{sync}): 
\begin{equation*}
g_n(a_n)=1, \quad dg_n(a_n)=0, \quad g_n(G_n^k(a_n))=0, \quad dg_n(G_n^k(a_n))=0, 
\end{equation*} 
for $1 \leq k \leq q_n-1$. Moreover, the estimate
\begin{equation}\label{estimgn} 
q_n^{-1}|g_n|_{\alpha,L_1}\leq N_{n}^{-2}, 
\end{equation}
holds true.
\end{proposition}  

The rest of this section is devoted to the proof of the above proposition. Note that together with the coupling lemma (Lemma~\ref{coupling}), this proposition easily gives a result of instability (analogous to Proposition 2.1 in \cite{MS02}) for a discrete system which is a perturbation of the map $\Phi^{\tilde{h}}$, but we prefer not to state such a result in order to focus on the continuous case. 

\paraga We first consider the simple pendulum
\[ P(\theta,I)=\demi I^2 + V(\theta), \quad (\theta,I)\in\A. \]
With our convention, the stable equilibrium is at $(0,0)$ and the unstable one is at $(0,1/2)$. Given any $M\in\N^*$, there is a unique point $b^M=(0,I_M)$ which is $M$-periodic for $\Phi^P$ (this is just the intersection between the vertical line $\{0\} \times \R$ and the closed orbit for the pendulum of period $M$). One can check that $I_M \in\, ]2,3[$ and as $M$ goes to infinity, $(0,I_M)$ tends to the point $(0,2)$ which belongs to the upper separatrix. Since $P_n(\theta,I)=\demi I^2 +N_n^{-2}V(\theta)$, then one can see that
\[ \Phi^{P_{n}}=(S_n)^{-1} \circ \Phi^{N_{n}^{-1}P} \circ S_{n}, \] 
where $S_{n}(\theta,I)=(\theta,N_n I)$ is the rescaling by $N_n$ in the action components. Therefore the point $b_n^M=(0,N_{n}^{-1}I_{M})$ is $q_n$-periodic for $\Phi^{P_{n}}$, for $q_n=N_nM$. Let $(\Phi_t^{P})_{t \in \R}$ be the flow of the pendulum, and 
\[ \Phi_t^{P}(0,I_{M})=(\theta_{M}(t),I_{M}(t)). \]
The function $\theta_{M}(t)$ is analytic. The crucial observation is the following simple property of the pendulum (see Lemma 2.2 in \cite{MS02} for a proof). 
\begin{figure}[t] 
\centering
\def\JPicScale{0.7}
\input{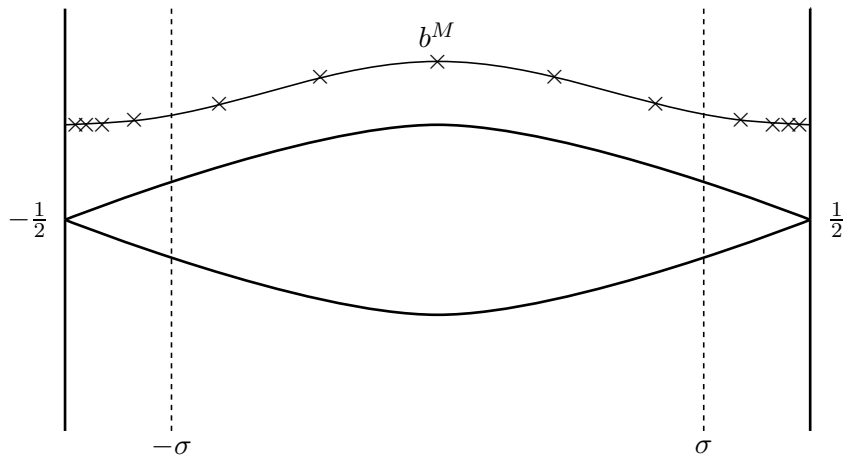}
\caption{The point $b^M$ and its iterates} \label{bN}
\end{figure}

\begin{lemma}
Let $\sigma=-\demi +\frac{2}{\pi}\arctan e^{\pi} < \frac{1}{2}$. For any $M \in \N^*$,  
\[ \theta_{M}(t) \notin [-\sigma,\sigma],  \]
for $t\in[1/2,M-1/2]$.
\end{lemma} 

Hence no matter how large $M$ is, most of the points of the orbit of $b^M\in\A$ will be outside the set $\{-\sigma \leq \theta \leq \sigma\}\times \R$ (see figure~\ref{bN}). The construction of a function that vanishes, as well as its first derivative, at these points, will be easily arranged by means of a function, depending only on the angle variables, with support in $\{-\sigma \leq \theta \leq \sigma\}$. 

As for the other points, it is convenient to introduce the function 
\[ \tau_{M} : [-\sigma,\sigma] \longrightarrow \, ]-1/2,1/2[ \] 
which is the analytic inverse of $\theta_{M}$. One can give an explicit formula for this map: 
\[ \tau_{M}(\theta)=\int_{0}^{\theta}\frac{d\varphi}{\sqrt{I_M^2-4\sin^2 \pi\varphi}}. \]
In particular, it is analytic and therefore it belongs to $G^{\alpha,L_1}([-\sigma,\sigma])$ for $\alpha\geq 1$ and $L_1>0$, and one can obtain the following estimate (see Lemma 2.3 in \cite{MS02} for a proof).

\begin{lemma}\label{lemmelambda}
For $\alpha>1$ and $L_1>0$, 
\[ \Lambda=\sup_{M\in\N^*}|\tau_M|_{\alpha,L_1}<+\infty. \]
\end{lemma}

Note that $\Lambda$ depends only on $\alpha$ and $L_1$. Under the action of $\tau_{M}$, the points of the orbit of $b^M$ whose projection onto $\T$ belongs to $\{-\sigma \leq \theta \leq \sigma\}$ get equi-distributed, and we can use the following elementary lemma.

\begin{lemma}\label{funct}
For $p \in \N^*$, the analytic function $\eta_p : \T \rightarrow \R$ defined by
\[ \eta_p(\theta)=\left(\frac{1}{p}\sum_{l=0}^{p-1}\cos 2\pi l\theta \right)^2 \]
satisfies
\[ \eta_p(0)=1, \quad \eta_p'(0)=0, \quad \eta_p(k/p)=\eta_p'(k/p)=0, \]
for $1 \leq k \leq p-1$, and 
\[ |\eta_p|_{\alpha,L_1}\leq e^{2\alpha L_1(2\pi p)^{\frac{1}{\alpha}}}. \]
\end{lemma}

The proof is trivial (see \cite{MS02}, Lemma 2.4). 

\paraga We can now pass to the proof of Proposition~\ref{pertu}.

\begin{proof}[Proof of Proposition~\ref{pertu}]
For $\alpha>1$ and $L_1>0$, consider the bump function $\varphi_{\alpha,L_1}\in G^{\alpha,L_1}(\T)$ given by Lemma~\ref{lemmeGev1} (see Appendix~\ref{Gev}). 

We choose our function $g_n\in G^{\alpha,L_1}(\A^{n-2})$, depending only on the angle variables, of the form
\[ g_n=g_n^{(2)}\otimes \cdots \otimes g_n^{(n-1)}, \]
where
\[ g_n^{(2)}(\theta_2)=\eta_{p_{n+3}}(\tau_{M_n}(\theta_2))\varphi_{\alpha,(4\sigma)^{-\frac{1}{\alpha}}L_1}((4\sigma)^{-1}\theta_2), \]
and
\[ g_n^{(i)}(\theta_i)=\eta_{p_{n+1+i}}(\theta_i), \quad 3\leq i \leq n-1.  \]
Let us write
\[ c_1=\left|\varphi_{\alpha,(4\sigma)^{-\frac{1}{\alpha}}L_1}\right|_{\alpha,(4\sigma)^{-\frac{1}{\alpha}}L_1}. \]
Now we choose our point $a_n=(a_n^{(2)},\dots,a_n^{(n-1)})\in\A^{n-2}$. We set
\[ a_n^{(2)}=b_n^{M_n}=(0,N_{n}^{-1}I_{M_n}),  \]
and
\[ a_n^{(i)}=(0,p_{n+1+i}^{-1}), \quad 3\leq i \leq n-1. \]
 
Let us prove that $a_n$ is $q_n$-periodic for $G_n$. We can write
\[ G_n=\Phi^{\frac{1}{2}I_2^2+N_{n}^{-2}V(\theta_2)}\times \Phi^{\frac{1}{2}(I_3^2+I_4^2+\cdots+I_{n-1}^2)}=\Phi^{P_n}\times \widehat{G}. \]
Since $p_{n+4}, \dots, p_{2n}$ are mutually prime, the point $(a_n^{(3)},\dots,a_n^{(n-1)})\in\A^{n-3}$ is periodic for $\widehat{G}$, with period 
\[N_n'=p_{n+4} \cdots p_{2n}.\]
By construction, the point $a_n^{(2)}=b_n^{M_n}\in\A$ is periodic for $\Phi^{P_n}$, with period $q_n=N_nM_n$, where
\[ N_n=p_{n+3}p_{n+4}\cdots p_{2n}.\]
This means that $a_n$ is periodic for the product map $G_n$, and the exact period is given by the least common multiple of $q_n$ and $N_n'$. Since $N_n'$ divides $q_n$, the period of $a_n$ is $q_n$.   

Now let us show that the synchronization conditions~(\ref{sync}) hold true, that is
\[ g_n(a_n)=1, \quad dg_n(a_n)=0, \quad g_n(G_n^k(a_n))=0, \quad dg_n(G_n^k(a_n))=0, \]
for $1\leq k\leq q_n-1$. Since $\varphi_{\alpha,L_1}(0)=1$, then
\[ g_n(a_n)=g_n^{(2)}(0)\cdots g_n^{(n-1)}(0)=1 \]
and as $\varphi_{\alpha,L_1}'(0)=0$, then 
\[ dg_n(a_n)=0. \]
To prove the other conditions, let us write $G_n^k(a_n)=(\theta_k,I_k)\in\A^{n-2}$, for $1\leq k\leq q_n-1$. 

If $\theta_k^{(2)}$ does not belong to $]-\sigma,\sigma[$, then $g_n^{(2)}$ and its first derivative vanish at $\theta_k^{(2)}$ because it is the case for $\varphi_{\alpha,(4\sigma)^{-\frac{1}{\alpha}}L}$, so 
\[ g_n(\theta_k)=dg_n(\theta_k)=0. \]
Otherwise, if $-\sigma<\theta_k^{(2)}<\sigma$, one can easily check that
\[ - \frac{N_n-1}{2}\leq k \leq \frac{N_n-1}{2} \]
and therefore
\[ \tau_{M_n}(\theta_k^{(2)})=\frac{k}{N_n},  \]
while
\[ \theta_k^{(i)}=\frac{k}{p_{n+i+1}}, \quad 3\leq i \leq n-1. \]
If $N_n'=p_{n+4} \cdots p_{2n}$ divides $k$, that is $k=k'N_n'$ for some $k'\in\Z$, then
\[ \tau_{M_n}(\theta_k^{(2)})=\frac{k}{N_n}=\frac{k'}{p_{n+3}} \]
and therefore, by Lemma~\ref{funct}, $\eta_{p_{n+3}}$ vanishes with its differential at $\theta_k^{(2)}$, and so does $g_n^{(2)}$. Otherwise, $N_n'$ does not divide $k$ and then, for $3\leq i \leq n-1$, at least one of the functions $\eta_{p_{n+1+i}}$ vanishes with its differential at $\theta_k^{(2)}$, and so does $g_n^{(i)}$. Hence in any case
\[ g_n(\theta_k)=dg_n(\theta_k)=0, \quad 1\leq k\leq q_n-1, \]
and the synchronization conditions~(\ref{sync}) are satisfied.

Now it remains to estimate the norm of the function $g_n$. First, using Lemma~\ref{lemmeGev2}, one finds 
\[ |g_n|_{\alpha,L_1} \leq \left|\varphi_{\alpha,(4\sigma)^{-\frac{1}{\alpha}}L_1}\right|_{\alpha,(4\sigma)^{-\frac{1}{\alpha}}L_1} |\eta_{p_{n+3}}\circ\tau_{M_n}|_{\alpha,L_1}|\eta_{p_{n+4}}|_{\alpha,L_1}|\eta_{p_{2n}}|_{\alpha,L_1},  \]
which by definition of $c_1$ gives
\[ |g_n|_{\alpha,L_1} \leq c_1|\eta_{p_{n+3}}\circ\tau_{M_n}|_{\alpha,L_1}|\eta_{p_{n+4}}|_{\alpha,L_1}|\eta_{p_{2n}}|_{\alpha,L_1}.  \]
Then, by definition of $\Lambda$ (Lemma~\ref{lemmelambda}) and using Lemma~\ref{lemmeGev3} (with $\Lambda_1=\Lambda^{\frac{1}{\alpha}}$), 
\[ |\eta_{p_{n+3}}\circ\tau_{M_n}|_{\alpha,L_1} \leq |\eta_{p_{n+3}}|_{\alpha,\Lambda^{\frac{1}{\alpha}}}, \]
so using Lemma~\ref{funct} and setting $c_2=2\alpha\sup\left\{\Lambda^{\frac{1}{\alpha}},L_1\right\}(2\pi)^{\frac{1}{\alpha}}$, this gives
\begin{eqnarray*}
|g_n|_{\alpha,L_1} & \leq & c_1 e^{2\alpha(\Lambda^{\frac{1}{\alpha}}+(n-3)L_1)(2\pi p_{2n})^{\frac{1}{\alpha}}} \\
& \leq & c_1 e^{2\alpha\sup\left\{\Lambda^{\frac{1}{\alpha}},L_1\right\}(n-2)(2\pi p_{2n})^{\frac{1}{\alpha}}} \\
& \leq & c_1 e^{c_2(n-2)p_{2n}^{\frac{1}{\alpha}}}.
\end{eqnarray*}
Finally, by definition of $M_n$ we obtain
\begin{equation*}
|g_n|_{\alpha,L_1}\leq M_nN_n^{-1},
\end{equation*}
and as $q_n=N_nM_n$, we end up with
\[ q_n^{-1}|g_n|_{\alpha,L_1}\leq N_n^{-2}. \]
This concludes the proof.
\end{proof}

\subsection{Proof of Theorem~\ref{thmnonpert}} \label{sectnonpert}

\paraga In the previous section, we were concerned with a perturbation of the integrable diffeomorphism $\Phi^{\tilde{h}}$, which can be written as $\Phi^u \circ \Phi^{\tilde{h}+v}$. So now we will briefly describe a suspension argument to go from this discrete case to a continuous case (we refer once again to \cite{MS02} for the details). 

Here we will make use of bump functions, however the process is still valid, though more difficult, in the analytic category, (see for example \cite{Dou88} or \cite{KP94}). The basic idea is to find a time-dependent Hamiltonian function on $\A^n$ such that the time-one map of its isotopy is $\Phi^u \circ \Phi^{\tilde{h}+v}$, or, equivalently, an autonomous Hamiltonian function on $\A^{n+1}$ such that its first return map to some $2n$-dimensional Poincaré section coincides with our map $\Phi^u \circ \Phi^{\tilde{h}+v}$. 

Given $\alpha>1$ and $L>1$, let us define the function
\[ \phi_{\alpha,L}=\left(\int_{\T}\varphi_{\alpha,L}\right)^{-1}\varphi_{\alpha,L}, \]
where $\varphi_{\alpha,L}$ is the bump function given by Lemma~\ref{lemmeGev1}. If $\phi_0(t)=\phi_{\alpha,L}\big(t-\frac{1}{4}\big)$ and $\phi_1(t)=\phi_{\alpha,L}\big(t-\frac{3}{4}\big)$, the time-dependent Hamiltonian
\[ H^*(\theta,I,t)=(\tilde{h}(I)+v(\theta)) \otimes \phi_0(t) + u(\theta) \otimes \phi_1(t) \]
clearly satisfies
\[ \Phi^{H^*}=\Phi^u \circ \Phi^{\tilde{h}+v}. \] 
But as $u$ and $v$ go to zero, $H^*$ converges to $\tilde{h} \otimes \phi_0$ rather than $\tilde{h}$. However, using classical generating functions, it is not difficult to modify the Hamiltonian in order to prove the following proposition (see Lemma 2.5 in \cite{MS02}).

\begin{proposition}[Marco-Sauzin]\label{sus}
Let $n\geq 1$, $R>1$, $\alpha>1$, $L_1>0$ and $L>0$ satisfying
\begin{equation}\label{LL1}
L_1^\alpha=L^\alpha(1+(L^\alpha+R+1/2)|\phi_{\alpha,L}|_{\alpha,L}). 
\end{equation}
If $u_n,v_n\in G^{\alpha,L_1}(\T^{n-1})$, there exists $f_n\in G^{\alpha,L}(\T^n\times \overline{B})$, independent of the variable $I_n$, such that if
\[ H_n(\theta,I)=\demi(I_1^2+\cdots+I_{n-1}^2)+I_n+f_n(\theta,I), \quad (\theta,I)\in\A^n, \]
for any energy $e\in\R$, the Poincaré map induced by the Hamiltonian flow of $H_n$ on the section $\{\theta_n=0\}\cap H_{n}^{-1}(e)$ coincides with the diffeomorphism 
\[ \Phi^{u_n}\circ\Phi^{\tilde{h}+v_n}.\] 
Moreover, one has
\begin{equation}\label{taillesus}
\sup\{|u_n|_{C^0},|v_n|_{C^0}\}\leq |f_n|_{\alpha,L} \leq c_3\sup\{|u_n|_{\alpha,L_1},|v_n|_{\alpha,L_1}\}, 
\end{equation}
where $c_3=2|\phi_{\alpha,L}|_{\alpha,L}$ depends only on $\alpha$ and $L$.
\end{proposition}

\paraga Now we can finally prove our theorem.

\begin{proof}[Proof of Theorem~\ref{thmnonpert}]
Let $R>1$, $\alpha>1$ and $L>0$, and choose $L_1$ satisfying the relation~(\ref{LL1}). The constants $c_1$ and $c_2$ of Proposition~\ref{pertu} depend only on $\alpha$ and $L_1$, hence they depend only on $R$, $\alpha$ and $L$. 

We can define $u_n, v_n\in G^{\alpha,L_1}(\T^{n-1})$ by
\[ u_n=q_{n}^{-1}U\otimes g_n,\quad v_n=N_{n}^{-2}V, \]
where $U(\theta_1)=-(2\pi)^{-1}\sin 2\pi\theta_1$, $V(\theta_2)=-\cos 2\pi\theta_2$ (so $v_n$ is formally defined on $\T$ but we identify it with a function on $\T^{n-1}$) and $g_n$ is the function given by Proposition~\ref{pertu}. Let us apply Proposition~\ref{sus}: there exists $f_n\in G^{\alpha,L}(\T^n\times \overline{B})$, independent of the variable $I_n$, such that if
\[ H_n(\theta,I)=\demi(I_1^2+\cdots+I_{n-1}^2)+I_n+f_n(\theta,I), \quad (\theta,I)\in\A^n, \]
for any energy $e\in\R$, the Poincaré map induced by the Hamiltonian flow of $H$ on the section $\{\theta_n=0\}\cap H^{-1}(e)$ coincides with the diffeomorphism 
\[ \Phi^{u_n}\circ\Phi^{\frac{1}{2}(I_1^2+\cdots+I_{n-1}^2)+v_n}=\Phi^{u_n}\circ\Phi^{\tilde{h}+v_n}.\] 
Let us show that our system $H_n$ has a drifting orbit. First consider its Poincaré section defined by
\[ \Psi_n=\Phi^{u_n}\circ\Phi^{\frac{1}{2}(I_1^2+\cdots+I_{n-1}^2)+v_n}=\Phi^{f_n \otimes g_n} \circ (F \times G_n),\]
with 
\[ f_n=q_n^{-1}U, \quad F=\Phi^{\demi I_1^2}, \quad G_n=\Phi^{\frac{1}{2}(I_2^2+I_3^2+\cdots+I_{n-1}^2)+N_{n}^{-2}V(\theta_2)}. \]
By Proposition~\ref{pertu}, we can apply the coupling lemma (Lemma~\ref{coupling}), so 
\[ \Psi_n^{q_n}((0,0),a_n)=(\Phi^f_n \circ F^{q_n}(0,0),a_n).  \]
Then, using~(\ref{driftstand2}), observe that
\[ \Phi^{f_n} \circ F^{q_n}=\Phi^{q_n^{-1}U} \circ \left(\Phi^{\frac{1}{2}I_1^2}\right)^{q_n} =\psi_{q_n}, \]
so 
\begin{eqnarray*}
\Psi_n^{q_n^2}((0,0),a_n) & = & ((\Phi^{f_n} \circ F^{q_n})^{q_n}(0,0),a_n) \\
& = & (\psi_{q_n}^{q_n}(0,0),a_n) \\
& = & ((0,1),a_n),
\end{eqnarray*}
where the last equality follows from~(\ref{driftstand}). Hence, after $q_n^2$ iterations, the $I_1$-component of the point $x_n=((0,0),a_n)\in\A^{n-1}$ drifts from $0$ to $1$. Then, for the continuous system, the initial condition $(x_n,t=0,I_n=0)$ in $\A^n$ gives rise to a solution $(x(t),t,I_n(t))=(x(t),\theta_n(t),I_n(t))$ of the Hamiltonian vector field generated by $H_n$ such that
\[ x(k)=\Psi_n^k(x_n), \quad k\in\Z. \] 
So after a time $\tau_n=q_n^2$, the point $(x_n,(0,0))$ drifts from $0$ to $1$ in the $I_1$-direction, and this gives our drifting orbit.

Now let $\varepsilon_n=|f_n|_{\alpha,L_1}$ be the size of our perturbation. Using the estimate~(\ref{estimgn}) and~(\ref{taillesus}) one finds
\begin{equation}\label{estNeps}
N_{n}^{-2}\leq \varepsilon_n \leq c_3 N_{n}^{-2}. 
\end{equation}
By the prime number theorem, $p_n$ is equivalent to $n\ln n$, so there exists $n_0\in\N^*$ such that for $n\geq n_0$, one can ensure that
\[ p_{2n}/4 \leq p_{n+i} \leq p_{2n}, \quad 3\leq i\leq n, \]
which gives
\begin{equation}\label{estimPN}
(p_{2n}/4)^{n-2} \leq N_n \leq p_{2n}^{n-2}, \quad N_n^{\frac{1}{n-2}}\leq p_{2n} \leq 4N_n^{\frac{1}{n-2}}. 
\end{equation}
We can also assume by the prime number theorem that for $n\geq n_0$, one has
\begin{equation}\label{nompremier}
2n\ln 2n \leq p_{2n} \leq 2(2n \ln 2n)=4n\ln 2n. 
\end{equation}
From the above estimates~(\ref{estimPN}) and~(\ref{nompremier}) one easily obtains
\begin{equation}\label{estN}
e^{(n-2)\ln (2^{-1}n\ln 2n)}\leq N_n\leq e^{(n-2)\ln (4n\ln 2n)},
\end{equation}
and, together with~(\ref{estNeps}), one finds
\begin{equation}\label{estpert}
e^{-2(n-2)\ln (4n\ln 2n)}\leq \varepsilon_n \leq c_3e^{-2(n-2)\ln (2^{-1}n\ln 2n)}.
\end{equation}

Finally it remains to estimate the time $\tau_n$. First recall that 
\[ M_n=2\left[c_1N_n e^{c_2(n-2)p_{2n}^{\frac{1}{\alpha}}}\right], \]
and with~(\ref{estimPN}) 
\[ q_n=N_nM_n  \leq  3c_1N_n^2 e^{4c_2(n-2)N_n^{\frac{1}{\alpha(n-2)}}}. \]
Hence
\[ q_n^2\leq  9c_1^2N_n^4 e^{8c_2(n-2)N_n^{\frac{1}{\alpha(n-2)}}}. \]
Then using~(\ref{estN}) we have
\[ N_n^{\frac{1}{\alpha(n-2)}}\leq (4n \ln 2n)^{\frac{1}{\alpha}} \] 
and from~(\ref{estNeps}) we know that
\[ N_n^4 \leq \left(\frac{c_3}{\varepsilon_n}\right)^2,  \]
so we obtain
\[ q_n^2\leq 9c_1^2\left(\frac{c_3}{\varepsilon_n}\right)^2 e^{8c_2(n-2)(4n \ln 2n)^{\frac{1}{\alpha}}}. \]
Now taking $n_0$ larger if necessary, as $\alpha>1$, one can ensure that for $n\geq n_0$,
\[ (4n)^{\frac{1}{\alpha}} \leq n, \quad (\ln 2n)^{\frac{1}{\alpha}}\leq \ln (2^{-1}n\ln 2n), \]
so 
\[ 8c_2(n-2)(4n \ln 2n)^{\frac{1}{\alpha}} \leq 8c_2(n-2)n\ln (2^{-1}n\ln 2n).  \]
Therefore
\begin{eqnarray*}
q_n^2 & \leq & 9c_1^2\left(\frac{c_3}{\varepsilon_n}\right)^2 e^{8c_2n(n-2)\ln (2^{-1}n\ln 2n)} \\
& \leq & 9c_1^2\left(\frac{c_3}{\varepsilon_n}\right)^2\left(e^{2(n-2)\ln (2^{-1}n\ln 2n))}\right)^{4c_2n}. 
\end{eqnarray*}
Finally by~(\ref{estpert}) we obtain
\begin{eqnarray*}
q_n^2 & \leq & 9c_1^2\left(\frac{c_3}{\varepsilon_n}\right)^2\left(\frac{c_3}{\varepsilon_n}\right)^{4c_2n} \\
& \leq & C\left(\frac{c}{\varepsilon_n}\right)^{n\gamma}
\end{eqnarray*}
with $C=9c_1^2$, $c=c_3$ and $\gamma=2+4c_2$. This ends the proof.
\end{proof}

\section{Proof of Theorem~\ref{thmnonpertana}}

The proof of Theorem~\ref{thmnonpertana} will be presented in section~\ref{sectnonpertana}, but first in section~\ref{mechanismana}, following \cite{LM05}, we will explain how the mechanism of instability that we explained in the Gevrey context can be (partly) generalized to an analytic context.

Recall that the first feature of the mechanism is to study perturbations of integrable maps and to obtain a result for perturbations of integrable flows by a ``quantitative" suspension argument. In the Gevrey case, this was particularly easy using compactly-supported functions. In the analytic case, this is more difficult but such a result exists, and here we will use a version due to Kuskin and Pöschel (\cite{KP94}). 

The second and main feature of the mechanism is the use of a coupling lemma, which enables us to embed a low-dimensional map having unstable orbits into a multi-dimensional near-integrable map. In the Gevrey case, we simply used the family of maps $\psi_q : \A \rightarrow \A$ defined as in~(\ref{maps1}) and the difficult part was the choice of the coupling, where we made an important use of the existence of compactly-supported functions. We do not know if this approach can be easily extended to the analytic case. However, by a result of Lochak and Marco (\cite{LM05}), one can still follow this path by using instead a suitable family of maps $\mathcal{F}_q : \A^2 \rightarrow \A^2$ having a well-controlled unstable orbit.   

\subsection{The modified mechanism}\label{mechanismana}

\paraga So let us describe this family of maps $\mathcal{F}_q : \A^2 \rightarrow \A^2$, $q\in\N^*$. We fix a width of analyticity $\sigma>0$ (to be chosen small enough in Proposition~\ref{LocMar} below). For $q$ large enough, $\mathcal{F}_q$ will appear as a perturbation of the following \textit{a priori} unstable map
\[ \mathcal{F}_*=\Phi^{\demi (I_1^2+I_2^2)+\cos 2\pi\theta_1} : \A^2 \rightarrow \A^2.  \]
More precisely, for $q\in\N^*$, let us define an analytic function $f_q : \A^2 \rightarrow \R$, depending only on the angle variables, by
\[ f_q(\theta_1,\theta_2)=f_{q}^{(1)}(\theta_1)f^{(2)}(\theta_2), \]
where
\[ f_{q}^{(1)}(\theta_1)=(\sin\pi\theta_1)^{\nu(q,\sigma)}, \quad f^{(2)}(\theta_2)=-\pi^{-1}(2+\sin 2\pi(\theta_2+6^{-1})). \]
We still have to define the exponent $\nu(q,\sigma)$ in the above expression for $f_{q}^{(1)}$. Let us denote by $[\,.\,]$ the integer part of a real number. Given $\sigma>0$, let $q_\sigma$ be the smallest positive integer such that 
\[ \left[\frac{\ln q_\sigma}{4\pi\sigma}+1\right]=1, \]
then we set
\[ \nu(q,\sigma)=2\left[\frac{\ln q_\sigma}{4\pi\sigma}+1\right], \quad q\geq q_\sigma. \]
In particular, for $q\geq q_\sigma$, $\nu(q,\sigma) \geq 2$ and it is always an even integer, hence $f^{(q)}$ is a well-defined $1$-periodic function. The reasons for the choice of this function $f^{(q)}$ are explained at length in \cite{Mar05} and \cite{LM05}, so we refer to these papers for some motivations.

Finally, for $q\geq q_\sigma$, we can define
\begin{equation}\label{Fq}
\mathcal{F}_q=\Phi^{q^{-1}f_q}\circ\mathcal{F}_* : \A^2 \rightarrow \A^2. 
\end{equation}
Let us also define the family of points $(\xi_{q,k})_{k\in\Z}$ of $\A^2$ by their coordinates
\[ \xi_{q,k} : (\theta_1=1/2, I_1=2, \theta_2=0, I_2=q^{-1}(k+1)). \]
Clearly, the $I_2$-component of the point $\xi_{q,k}$ converges to $\pm \infty$ when $k$ goes to $\pm \infty$, hence the sequence $(\xi_{q,k})_{k\in\Z}$ is wandering in $\A^2$.

\paraga The following result was proved in \cite{LM05}, Proposition 2.1.

\begin{proposition}[Lochak-Marco]\label{LocMar}
There exist a width $\sigma>0$, an integer $q_0$ and a constant $0<d<1$ such that for any $q\geq q_0$, the diffeomorphism $\mathcal{F}_q : \A^2 \rightarrow \A^2$ has a point $\zeta_q \in \A^2$ which satisfies
\begin{equation*}
|\mathcal{F}_{q}^{kq}(\zeta_q)-\xi_{q,k}|\leq d^{\nu(q,\sigma)}.
\end{equation*}
\end{proposition}

As a consequence, the orbit of the point $\zeta_q \in \A^2$ under the map $\mathcal{F}_{q}$ is also wandering in $\A^2$. In particular, for $k=0$ and $k=3q$ the above estimate yields
\[ |\zeta_q-\xi_{q,0}|\leq d^{\nu(q,\sigma)}, \quad |\mathcal{F}_{q}^{3q^2}(\zeta_q)-\xi_{q,3q}|\leq d^{\nu(q,\sigma)}, \] 
and as
\[ |\xi_{q,3q}-\xi_{q,0}|=3, \]
one obtains
\begin{equation}\label{timeana}
|\mathcal{F}_{q}^{3q^2}(\zeta_q)-\zeta_q| \geq 3-2d^{\nu(q,\sigma)} \geq 1.
\end{equation}

The proof of the above proposition is rather difficult and it would be too long to explain it. We just mention that crucial ingredients are on the one hand a conjugacy result for normally hyperbolic manifolds (in the spirit of Sternberg) adapted to this analytic and symplectic context, and on the other hand the classical method of correctly aligned windows introduced by Easton. 

\paraga Now this family of maps $\mathcal{F}_q : \A^2 \rightarrow \A^2$ will be used in the coupling lemma. More precisely, recalling the notations of the coupling lemma~\ref{coupling}, in the following we shall take $m=2$, 
\[ F=F_n=\Phi^{\demi (I_1^2+I_2^2)+N_n^{-2}\cos 2\pi\theta_1},\] 
which is just a rescaled version of the map $\mathcal{F}_*$ we introduced before, and $f=f_n=q_{n}^{-1}f_{q_n}$, for some positive integer parameters $N_n$ and $q_n$ to be defined below.

It remains to choose the dynamics on the second factor, and here it will be an easy task. In order to have a result for a continuous system with $n$ degrees of freedom, we set $m'=n-3$, and it will be just fine to take
\[ G=G_n=\Phi^{\demi(I_3^2+\cdots+I_{n-1}^{2})}. \]
If $(p_j)_{j\geq 0}$ is the ordered sequence of prime numbers, now we let $N_n$ be the product of the $n-3$ prime numbers $\{p_{n+4},\dots,p_{2n}\}$, that is
\begin{equation}\label{Nn}
N_n=p_{n+4}p_{n+5}\dots p_{2n}. 
\end{equation}
The next proposition is the analytic analogue of Proposition~\ref{pertu}, and its proof is even simpler.

\begin{proposition}\label{pertuana}
Let $n\geq 4$ and $\sigma >0$. Then there exist a function $g_n\in \mathcal{A}_{\sigma}(\T^{n-3})$ and a point $a_n\in\A^{n-3}$ such $a_n$ is $N_n$-periodic for $G_n$ and $(g_n, G_n, a_n, N_n)$ satisfy the synchronization conditions (\ref{sync}): 
\begin{equation*}
g_n(a_n)=1, \quad dg_n(a_n)=0, \quad g_n(G_n^k(a_n))=0, \quad dg_n(G_n^k(a_n))=0, 
\end{equation*} 
for $1 \leq k \leq N_n-1$. Moreover, there exists a positive constant $c$ depending only on $\sigma$ such that if
\begin{equation}\label{qn}
q_n=2N_n^4[e^{c(n-3)p_{2n}}],
\end{equation}
the estimate
\begin{equation}\label{estimgnana} 
q_n^{-1/2}|g_n|_{\sigma}\leq N_{n}^{-2}, 
\end{equation}
holds true.
\end{proposition}

The function $g_n$ belongs to $\mathcal{A}_{\sigma}(\T^{n-3})$, but it can also be considered as a function in $\mathcal{A}_{\sigma}(\A^{n-3})$ depending only on the angle variables.

As in the previous section, one can easily see that the coupling lemma, together with both Proposition~\ref{LocMar} and Proposition~\ref{pertuana}, already give us a result of instability for a perturbation of an integrable map, but we shall not state it. 

\begin{proof}
Recall that for $p \in \N^*$, we have defined in Lemma~\ref{funct} an analytic function $\eta_p : \T \rightarrow \R$ by
\[ \eta_p(\theta)=\left(\frac{1}{p}\sum_{l=0}^{p-1}\cos 2\pi l\theta \right)^2. \]
We choose our function $g_n\in \mathcal{A}_{\sigma}(\T^{n-3})$ of the form
\[ g_n=g_n^{(3)}\otimes \cdots \otimes g_n^{(n-1)}, \]
where
\[ g_n^{(i)}(\theta_i)=\eta_{p_{n+1+i}}(\theta_i), \quad 3\leq i \leq n-1,  \]
and our point $a_n=(a_n^{(3)},\dots,a_n^{(n-1)})\in\A^{n-3}$ where
\[ a_n^{(i)}=(0,p_{n+1+i}^{-1}), \quad 3\leq i \leq n-1. \]
Recalling the definition of $G_n$ and $N_n$, it is obvious that $a_n$ is $N_n$-periodic for $G_n$. Moreover, by Lemma~\ref{funct}, the function $\eta_p$ satisfies
\[ \eta_p(0)=1, \quad \eta_p'(0)=0, \quad \eta_p(k/p)=\eta_p'(k/p)=0, \]
for $1 \leq k \leq p-1$, from which one can easily deduce that $(g_n, G_n, a_n, N_n)$ satisfy the synchronization conditions (\ref{sync}).

Concerning the estimate, first note that
\[ |\eta_p|_{\sigma}\leq e^{4\pi\sigma p} \]
so that
\[ |g_n|_{\sigma} \leq |\eta_{p_{n+4}}|_{\sigma}\cdots |\eta_{p_{2n}}|_{\sigma} \leq e^{4\pi \sigma(n-3)p_{2n}}. \]
Therefore, if we set $c=8\pi\sigma$, then by definition of $q_n$ one has
\[ q_n^{1/2} \geq N_n^2 e^{4\pi \sigma(n-3)p_{2n}}\]
and this eventually gives us
\[ q_n^{-1/2}|g_n|_{\sigma}\leq N_{n}^{-2}, \]
which is the desired estimate.  
\end{proof}

\subsection{Proof of Theorem~\ref{thmnonpertana}}\label{sectnonpertana}

\paraga First we shall recall the following result of Kuksin-Pöschel (\cite{KP94}, see also \cite{Kuk93}). 

\begin{proposition}[Kuksin-Pöschel]\label{susana}
Let $\Psi_n : \A^{n-1} \rightarrow \A^{n-1}$ be a bounded real-analytic exact-symplectic diffeomorphism, which has a bounded holomorphic extension to some complex neighbourhood $V_{\varrho}$, for some width $\varrho>0$ independent of $n\in\N^*$. Assume also that $|\Psi_n-\Phi^{\tilde{h}}|_{\varrho}$ goes to zero when $n$ goes to infinity, where $\tilde{h}(I_1,\dots,I_{n-1})=\demi(I_1^2+\cdots+I_{n-1}^2)$.

Then there exist $n_0\in\N^*$, $\rho<\varrho$ such that for any $n\geq n_0$, there exists $f_n\in \mathcal{A}_\rho(\T^n \times B)$, independent of the variable $I_n$, such that if
\[ H_n(\theta,I)=\demi(I_1^2+\cdots+I_{n-1}^2)+I_n+f_n(\theta,I), \quad (\theta,I)\in\A^n, \]
for any energy $e\in\R$, the Poincaré map induced by the Hamiltonian flow of $H_n$ on the section $\{\theta_n=0\}\cap H_{n}^{-1}(e)$ coincides with $\Psi_n$. Moreover, the estimate
\begin{equation}\label{taillesusana}
|\Psi_n-\Phi^{\tilde{h}}|_{\varrho} \leq |f_n|_\rho \leq \delta_n |\Psi_n-\Phi^{\tilde{h}}|_{\varrho}, 
\end{equation}
holds true for some constant $\delta_n$ that may depends on $n\in\N^*$.
\end{proposition} 

This suspension result is slightly less accurate (since more difficult) than Proposition~\ref{sus}, as there is a constant $\delta_n$ depending on $n$. However, what really matters is that the resulting width of analyticity $\rho$ depends only on $\varrho$ and $R$, but not on $n$.  

\paraga Now we can finally prove the theorem.

\begin{proof}[Proof of Theorem~\ref{thmnonpertana}]
Let $n\geq 4$, $R>1$ and $\sigma>0$ given by the Proposition~\ref{LocMar}, and let $N_n$ and $q_n$ defined as in (\ref{Nn}) and (\ref{qn}) respectively. 

We will first construct a map $\Psi_n$ with a well-controlled wandering point. To this end, by Proposition~\ref{pertuana} we can apply the coupling lemma~\ref{coupling} with the following data:
\[ F_n=\Phi^{\demi (I_1^2+I_2^2)+N_n^{-2}\cos 2\pi\theta_1}, \quad f_n=q_{n}^{-1}f_{q_n}, \quad G_n=\Phi^{\demi(I_3^2+\cdots+I_{n-1}^{2})}, \]
and with the function $g_n$ and the point $a_n$ given by the aforementioned proposition. This gives us the following: if
\[ u_n=q_{n}^{-1} f_{q_n}\otimes g_n , \quad v_n=N_n^{-2}V  \]
where $V(\theta_1)=\cos 2\pi\theta_1$, then the $N_n$-iterates of the map
\[ \Psi_n=\Phi^{u_n} \circ \Phi^{\tilde{h}+v_n} : \A^{n-1} \rightarrow \A^{n-1}\]
satisfies the following relation:
\begin{equation}\label{coup}
\Psi_n^{N_n}(x,a_n)=(\Phi^{q_{n}^{-1}f_{q_n}}\circ F_n^{N_n}(x),a_n), \quad x\in \A^2. 
\end{equation}
Now let us look at the map
\[ \Phi^{q_{n}^{-1}f_{q_n}}\circ F_n^{N_n}=\Phi^{q_{n}^{-1}f_{q_n}}\circ\left(\Phi^{\demi (I_1^2+I_2^2)+N_n^{-2}\cos 2\pi\theta_1}\right)^{N_n}. \]
If $S_n(\theta_1,\theta_2,I_1,I_2)=(\theta_1,\theta_2,N_nI_1,N_nI_2)$ is the rescaling by $N_n$ in the action components, one sees that
\[ \Phi^{q_{n}^{-1}f_{q_n}}\circ F_n^{N_n}=S_n^{-1}\circ \mathcal{F}_{N_n^{-1}q_n}\circ S_n \]
where $\mathcal{F}_{N_n^{-1}q_n}$ is defined in~(\ref{Fq}). Now by Proposition~\ref{LocMar}, choosing $n$ large enough so that $N_n^{-1}q_n\geq q_0$, this map has a wandering point $\zeta_{N_n^{-1}q_n}\in\A^2$, which by~(\ref{timeana}) satisfies
\[ \left|\mathcal{F}_{N_n^{-1}q_n}^{3N_n^{-2}q_n^2} \left(\zeta_{N_n^{-1}q_n}\right)-\zeta_{N_n^{-1}q_n}\right| \geq 1. \]
Using the above conjugacy relation, one finds that the point
\[ \chi_n=S_n^{-1}(\zeta_{N_n^{-1}q_n})\in\A^2 \]
wanders under the iteration of $\Phi^{q_{n}^{-1}f_{q_n}}\circ F_n^{N_n}$, and that its drift is bigger than one after $N_n(3N_n^{-2}q_n^2)=3N_n^{-1}q_n^2$ iterations, that is
\[ \left|(\Phi^{q_{n}^{-1}f_{q_n}}\circ F_n^{N_n})^{3N_n^{-1}q_n^2}(\chi_n)-\chi_n\right|\geq 1. \]  
By the relation~(\ref{coup}) this gives a wandering point $x_n=(\chi_n,a_n)\in\A^{n-1}$ for the map $\Psi_n$, satisfying the estimate
\begin{equation}\label{estimpsi}
|\Psi_n^{3q_n^2}(x_n)-x_n|\geq 1. 
\end{equation}

Next let us estimate the distance between $\Psi_n$ and the integrable diffeomorphism $\Phi^{\tilde{h}}$. First note that since $u_n,v_n\in \mathcal{A}_\sigma(\T^{n-1})$, $\Psi_n$ extends holomorphically to a complex neighbourhood of size $\sigma$. Let us now estimate the norms of $u_n$ and $v_n$. Obviously, one has
\[ N_n^{-2}\leq|v_n|_\sigma\leq e^{2\pi\sigma}N_n^{-2}. \]
By definition of $f_q$ and the exponent $\nu(q,\sigma)$, one easily obtains
\[ |q_n^{-1}f_{q_n}|_\sigma \leq q_n^{-1/2}|f^{(2)}|_\sigma, \]
and hence
\[ |u_n|_\sigma \leq |q_n^{-1}f_{q_n}|_\sigma |g_n|_\sigma \leq  q_n^{-1/2} |g_n|_\sigma |f^{(2)}|_\sigma \leq N_n^{2} |f^{(2)}|_\sigma, \]
where the last inequality follows from the estimate~(\ref{estimgnana}). Then by using Cauchy estimates and general inequalities on time-one maps, we obtain
\begin{equation}\label{estimdis}
N_n^{-2} \leq |\Psi_n-\Phi^{\tilde{h}}|_{\varrho} \leq c_\sigma N_n^{-2}, 
\end{equation}
for $n$ large enough, and for some constants $c_\sigma$ and $\varrho>0$ depending only on $\sigma$ (for instance, one can choose $\varrho=6^{-1}\sigma$).

Now we can eventually apply Proposition~\ref{susana}: there exist $n_0\in\N^*$, $\rho<\varrho$ such that for any $n\geq n_0$, there exists $f_n\in \mathcal{A}_\rho(\T^n \times B)$, independent of the variable $I_n$, such that if
\[ H_n(\theta,I)=\demi(I_1^2+\cdots+I_{n-1}^2)+I_n+f_n(\theta,I), \quad (\theta,I)\in\A^n, \]
for any energy $e\in\R$, the Poincaré map induced by the Hamiltonian flow of $H_n$ on the section $\{\theta_n=0\}\cap H_{n}^{-1}(e)$ coincides with $\Psi_n$. Clearly, the wandering point $x_n$ for $\Psi_n$ gives us a wandering orbit $(x(t),t,I_n(t))=(x(t),\theta_n(t),I_n(t))$ for the Hamiltonian vector field generated by $H_n$, such that
\[ x(k)=\Psi_n^k(x_n), \quad k\in\Z. \]
In particular, after a time $\tau_n=3q_n^2$, by the above equality and the relation~(\ref{estimpsi}) this orbit drifts from $0$ to $1$.

Now it remains to estimate the size of the perturbation $\varepsilon_n=|f_n|_\rho$ and the time of drift $\tau_n$ in terms of the number of degrees of freedom $n$. First, by~(\ref{taillesusana}) and~(\ref{estimdis}),
\begin{equation}\label{estNepsana}
N_{n}^{-2}\leq \varepsilon_n \leq c_n N_{n}^{-2}, 
\end{equation}
with $c_n=c_\sigma \delta_n$. Then, by the prime number theorem, taking $n_0$ large enough, one can ensure that
\[ p_{2n}/4 \leq p_{n+i} \leq p_{2n}, \quad 4\leq i\leq n, \]
which gives
\begin{equation}\label{estimPNana}
(p_{2n}/4)^{n-3} \leq N_n \leq p_{2n}^{n-3}, \quad N_n^{\frac{1}{n-3}}\leq p_{2n} \leq 4N_n^{\frac{1}{n-3}}. 
\end{equation}
We can also assume by the prime number theorem that for $n\geq n_0$, one has
\begin{equation}\label{nompremierana}
2n\ln 2n \leq p_{2n} \leq 2(2n \ln 2n)=4n\ln 2n. 
\end{equation}
From the above estimates~(\ref{estimPNana}) and~(\ref{nompremierana}) one easily obtains
\begin{equation}\label{estNana}
e^{(n-3)\ln (2^{-1}n\ln 2n)}\leq N_n\leq e^{(n-3)\ln (4n\ln 2n)},
\end{equation}
and, together with~(\ref{estNepsana}), one finds
\begin{equation}\label{estpertana}
e^{-2(n-3)\ln (4n\ln 2n)}\leq \varepsilon_n \leq c_n e^{-2(n-3)\ln (2^{-1}n\ln 2n)}.
\end{equation} 
Concerning the time $\tau_n$, we have
\[ \tau_n=3q_n^2 \leq 12 N_n^{8} e^{2c(n-3)p_{2n}} \leq 12 N_n^{8} e^{8c(n-3)N_n^{\frac{1}{n-3}}} , \]
where the last inequality follows from~(\ref{estimPNana}). Then using~(\ref{estNana}) we have
\[ N_n^{\frac{1}{(n-3)}}\leq 4n \ln 2n \] 
and from~(\ref{estNepsana}) we know that
\[ N_n^8 \leq \left(\frac{c_n}{\varepsilon_n}\right)^4,  \]
so we obtain
\[ q_n^2\leq 12\left(\frac{c_3}{\varepsilon_n}\right)^4 e^{32c(n-3)n \ln 2n}. \]
Then one can ensure that for $n\geq n_0$,
\[ \ln 2n \leq \ln (2^{-1}n\ln 2n), \]
so 
\[ 32c(n-3)n \ln 2n \leq 32c(n-3)n \ln (2^{-1}n\ln 2n).  \]
Therefore
\begin{eqnarray*}
q_n^2 & \leq & 12\left(\frac{c_n}{\varepsilon_n}\right)^4 e^{32c(n-3)n \ln (2^{-1}n\ln 2n)} \\
& \leq & 12\left(\frac{c_n}{\varepsilon_n}\right)^4\left(e^{2(n-3)\ln (2^{-1}n\ln 2n))}\right)^{16 cn}. 
\end{eqnarray*}
Finally by~(\ref{estpertana}) we obtain
\begin{eqnarray*}
q_n^2 & \leq & 12\left(\frac{c_n}{\varepsilon_n}\right)^4\left(\frac{c_n}{\varepsilon_n}\right)^{16cn} \\
& \leq & C\left(\frac{c_n}{\varepsilon_n}\right)^{n\gamma}
\end{eqnarray*}
with $C=12$ and $\gamma=4+16c$. This concludes the proof.
\end{proof}

\appendix

\section{Gevrey functions}\label{Gev}

In this very short appendix, we recall some facts about Gevrey functions that we used in the text. We refer to~\cite{MS02}, Appendix A, for more details.

The most important property of $\alpha$-Gevrey functions is the existence, for $\alpha>1$, of bump functions.

\begin{lemma}\label{lemmeGev1}
Let $\alpha>1$ and $L>0$. There exists a non-negative $1$-periodic function $\varphi_{\alpha,L}\in G^{\alpha,L}\left([-\frac{1}{2},\frac{1}{2}]\right)$ whose support is included in $[-\frac{1}{4},\frac{1}{4}]$ and such that $\varphi_{\alpha,L}(0)=1$ and $\varphi_{\alpha,L}'(0)=0$.
\end{lemma}

The following estimate on the product of Gevrey functions follows easily from the Leibniz formula. 

\begin{lemma}\label{lemmeGev2}
Let $L>0$, and $f,g\in G^{\alpha,L}(\T^n \times \overline{B})$. Then
\[ |fg|_{\alpha,L}\leq |f|_{\alpha,L}|g|_{\alpha,L}. \]
\end{lemma}

Finally, estimates on the composition of Gevrey functions are much more difficult (see Proposition A.1 in \cite{MS02}), but here we shall only need the following statement.

\begin{lemma}\label{lemmeGev3}
Let $\alpha\geq 1$, $\Lambda_1>0, L_1>0$, and $I,J$ be compact intervals of $\R$. Let $f\in G^{\alpha,\Lambda_1}(I)$, $g\in G^{\alpha,L_1}(J)$ and assume $g(J)\subseteq I$. If
\[ |g|_{\alpha,L_1}\leq \Lambda_{1}^{\alpha},\] 
then $f\circ g \in G^{\alpha,L_1}(J)$ and
\[ |f\circ g|_{\alpha,L_1} \leq |f|_{\alpha,\Lambda_1}. \]
\end{lemma}

{\it Acknowledgments.} 

The author is indebted to Jean-Pierre Marco for suggesting him this problem to work on, for helpful discussions, comments and corrections on a first version of this paper. He also thanks the anonymous referee for several interesting suggestions. Finally, the author thanks the University of Warwick where this work has been finished while he was a Research Fellow through the Marie Curie training network ``Conformal Structures and Dynamics (CODY)".

\addcontentsline{toc}{section}{References}
\bibliographystyle{amsalpha}
\bibliography{NonPert3}

\newcommand{\etalchar}[1]{$^{#1}$}
\providecommand{\bysame}{\leavevmode\hbox to3em{\hrulefill}\thinspace}
\providecommand{\MR}{\relax\ifhmode\unskip\space\fi MR }
\providecommand{\MRhref}[2]{%
  \href{http://www.ams.org/mathscinet-getitem?mr=#1}{#2}
}
\providecommand{\href}[2]{#2}
\begin{thebibliography}{CKG{\etalchar{+}}10}

\bibitem[BK05]{BK05}
J.~Bourgain and V.~Kaloshin, \emph{{On diffusion in high-dimensional
  {H}amiltonian systems}}, J. Funct. Anal. \textbf{229} (2005), no.~1, 1--61.

\bibitem[BM10]{BM10}
A.~Bounemoura and J.-P. Marco, \emph{Improved exponential stability for
  quasi-convex {H}amiltonians}, Preprint (2010).

\bibitem[CKG{\etalchar{+}}10]{CKSTT}
J.~Colliander, M.~Keel, Staffilani G., H.~Takaoka, and T.~Tao, \emph{Transfer
  of energy to high frequencies in the cubic defocusing nonlinear {S}chrödinger
  equation}, Invent. Math. \textbf{181} (2010), no.~1, 39--113.

\bibitem[Dou88]{Dou88}
R.~Douady, \emph{Stabilité ou instabilité des points fixes elliptiques}, Ann.
  Sci. Ec. Norm. Sup. \textbf{21} (1988), no.~1, 1--46.

\bibitem[GG10]{GG10}
P.~Gérard and S.~Grellier, \emph{The cubic {S}zegö equation}, Ann. Sci. Ec.
  Norm. Sup. \textbf{43} (2010), no.~5, 761--810.

\bibitem[KP94]{KP94}
S.~Kuksin and J.~Pöschel, \emph{On the inclusion of analytic symplectic maps in
  analytic {H}amiltonian flows and its applications}, Seminar on dynamical
  systems (1994), 96--116, Birkhäuser, {B}asel.

\bibitem[Kuk93]{Kuk93}
S.~Kuksin, \emph{On the inclusion of an almost integrable analytic
  symplectomorphism into a {H}amiltonian flow}, Russian J. Math. Phys.
  \textbf{1} (1993), no.~2, 191--207.

\bibitem[LM05]{LM05}
P~Lochak and J.P. Marco, \emph{Diffusion times and stability exponents for
  nearly integrable analytic systems}, Central European Journal of Mathematics
  \textbf{3} (2005), no.~3, 342--397.

\bibitem[LN92]{LN92}
P.~Lochak and A.I. Neishtadt, \emph{Estimates of stability time for nearly
  integrable systems with a quasiconvex {H}amiltonian}, Chaos \textbf{2}
  (1992), no.~4, 495--499.

\bibitem[Mar05]{Mar05}
J.-P. Marco, \emph{Uniform lower bounds of the splitting for analytic
  symplectic systems}, preprint (2005).

\bibitem[MS02]{MS02}
J.-P. Marco and D.~Sauzin, \emph{Stability and instability for {G}evrey
  quasi-convex near-integrable {H}amiltonian systems}, Publ. Math. Inst. Hautes
  Études Sci. \textbf{96} (2002), 199--275.

\bibitem[MS04]{MS04}
\bysame, \emph{Wandering domains and random walks in {G}evrey near-integrable
  systems}, Erg. Th. Dyn. Sys. \textbf{5} (2004), 1619--1666.

\bibitem[Nek77]{Nek77}
N.N. Nekhoroshev, \emph{An exponential estimate of the time of stability of
  nearly integrable {H}amiltonian systems}, Russian Math. Surveys \textbf{32}
  (1977), no.~6, 1--65.

\bibitem[Nek79]{Nek79}
\bysame, \emph{An exponential estimate of the time of stability of nearly
  integrable {H}amiltonian systems {II}}, Trudy Sem. Petrovs \textbf{5} (1979),
  5--50.

\bibitem[Nie96]{Nie96}
L.~Niederman, \emph{Stability over exponentially long times in the planetary
  problem}, Nonlinearity \textbf{9} (1996), no.~6, 1703--1751.

\bibitem[Pös93]{Pos93}
J.~Pöschel, \emph{Nekhoroshev estimates for quasi-convex {H}amiltonian
  systems}, Math. Z. \textbf{213} (1993), 187--216.

\bibitem[Zha09]{KZ09}
Ke~Zhang, \emph{Speed of {A}rnold diffusion for analytic {H}amiltonian
  systems}, Preprint (2009).

\end{thebibliography}

\end{document}